\documentclass[12pt]{amsart}
\usepackage{scalefnt}
\newtheorem{thr}{Theorem}[section]
\newtheorem{lem}[thr]{Lemma}
\newtheorem{cor}[thr]{Corollary}
\newtheorem{stat}[thr]{Proposition}

\theoremstyle{definition}
\newtheorem{defn}[thr]{Definition}

\newtheorem{quest}[thr]{Question}
\newtheorem{prob}[thr]{Problem}

\theoremstyle{remark}

\newtheorem{observ}[thr]{Observation}

\numberwithin{equation}{section}

\def\i{\infty}

\def\op{\oplus}
\def\ot{\otimes}

\def\R{\mathbf{R}}
\def\Z{\mathbf{Z}}

\def\Ro{\overline{\R}}

\def\A{{\mathcal A}}
\def\B{{\mathcal B}}
\def\C{{\mathcal C}}
\def\D{{\mathcal D}}

\def\Hh{{\mathcal H}}
\def\Q{{\mathcal Q}}
\def\Tt{{\mathcal T}}
\def\Vv{{\mathcal V}}
\def\vv{{^{ver}}}
\def\hh{{^{head}}}
\def\tt{{^{tail}}}

\begin{document}


\title[The complexity of tropical matrix factorization]{The complexity of tropical matrix factorization}

\author[Yaroslav Shitov]{Yaroslav Shitov}
\address{National Research University Higher School of Economics, 20 Myasnitskaya Ulitsa, Moscow 101000, Russia}
\email{yaroslav-shitov@yandex.ru}


\begin{abstract}
The tropical arithmetic operations on $\R$ are defined by $a\op b=\min\{a,b\}$ and $a\ot b=a+b$.
Let $A$ be a tropical matrix and $k$ a positive integer, the problem of Tropical Matrix Factorization (TMF) asks
whether there exist tropical matrices $B\in\R^{m\times k}$ and $C\in\R^{k\times n}$ satisfying $B\ot C=A$.
We show that no algorithm for TMF is likely to work in polynomial time for every fixed $k$,
thus resolving a problem proposed by Barvinok in 1993. 
\end{abstract}

\maketitle

\section{Introduction}

The \textit{tropical semiring} is the set $\R$ of real numbers equipped with the operations of tropical addition
and tropical multiplication, which are defined by $a\op b=\min\{a,b\}$, $a\ot b=a+b$.
The tropical semiring is essentially the same structure as the \textit{max-plus algebra},
which is the set $\R$ with the operations of maximum and sum, and is being studied since the 1960's,
when the applications in the optimization theory have been found~\cite{Vor}.
The tropical arithmetic operations on $\R$, which allow us to formulate a number of important non-linear
problems in a linear-like way, arise indeed in a variety of topics in pure and applied mathematics.
The study of tropical mathematics has applications in operations research~\cite{CG},
discrete event systems~\cite{BCOQ}, automata theory~\cite{Sim}, optimal control~\cite{KM, McE},
algebraic geometry~\cite{DSS, EKL}, and others; we refer to~\cite{GP2} for a detailed survey of applications.
A considerable number of important problems in tropical mathematics has a linear-algebraic nature.
For instance, the concepts of eigenvalue and eigenvector, the theory of linear systems,
and the algorithms for computing rank functions are useful for different applications~\cite{AGG, DSS, HOW, GP2}.
Some applications also give rise to studying the multiplicative structure of tropical matrices~\cite{Pin},
and in this context, the Burnside-type problems are important~\cite{Gau, Sim}.
Another interesting problem is to study the subgroup structure of the semigroup of tropical matrices under multiplication~\cite{IJK, JK2}.

In our paper, we consider the problem of matrix factorization, which is also related to the concept of factor rank
of matrices over semirings~\cite{BKS}. The study of factor rank dates back to the 1980's~\cite{BP2}, and has now
numerous applications in different contexts of mathematics. Being considered on the semiring of nonnegative matrices,
the factor rank is known as nonnegative rank and has applications in quantum mechanics, statistics, demography,
and others~\cite{CR}. The factor rank of matrices over the binary Boolean semiring is also called Boolean rank
and has applications in combinatorics and graph theory~\cite{Bea, GP}. Finally, for matrices over a field, the factor rank
coincides with the classical rank function.

In the context of matrices over the tropical semiring, the factor rank is also known as
\textit{combinatorial rank}~\cite{Barv2} and \textit{Barvinok rank}~\cite{DSS},
and the study of this notion has arisen from combinatorial optimization~\cite{Barv}.
The factor rank appears in the formulation of a number of problems in optimization, for instance, in the Traveling
Salesman problem with warehouses~\cite{Barv3}.
Also, the notion of factor rank is of interest in the study of tropical geometry~\cite{Dev2, DSS},
where the factor rank can be thought of as the minimum number of points whose tropical convex hull contains the columns of a matrix.
Let us define the factor rank function for tropical matrices, assuming that multiplication of tropical matrices
is understood as ordinary matrix multiplication with $+$ and $\cdot$ replaced by the tropical operations $\op$ and $\ot$.

\begin{defn}\label{factrank}
The \textit{factor rank} of a tropical matrix $A\in\R^{m\times n}$ is the smallest integer $k$ for which there exist
tropical matrices $B\in\R^{m\times k}$ and $C\in\R^{k\times n}$ satisfying $B\ot C=A$.
\end{defn}

The most straightforward way of computing the factor rank is based on the quantifier elimination algorithm
for the theory of reals with addition and order~\cite{FR}. Indeed, Definition~\ref{factrank} allows us to
define the set of all $m$-by-$n$ matrices with factor rank $k$ by a first-order order formula. We then
employ the decision procedure based on the quantifier elimination algorithm provided in~\cite{FR} to
check whether a given $m$-by-$n$ matrix indeed has factor rank $k$.

However, the computational complexity of quantifier elimination makes the algorithm
mentioned unacceptable for practical use. Another algorithm for computing
the factor rank is given by Develin in the paper~\cite{Dev2}, where he develops
the theory of tropical secant varieties. He characterizes the factor rank from
a point of view of tropical geometry, and the characterization obtained provides
an algorithm for computing the factor rank.

Unfortunately, neither the algorithm by Develin nor any other algorithm is likely to compute the factor rank
of a tropical matrix in polynomial time.
Indeed, the problem of computing the factor rank is NP-hard even for tropical $01$-matrices, see~\cite{DSS}.
In other words, the general problem of \textit{Tropical Matrix Factorization} (\textit{TMF}),
which asks whether a given matrix $A$ and a given integer $k$ are such that $A=B\ot C$ for some $B\in\R^{m\times k}$ and
$C\in\R^{k\times n}$, turns out to be NP-hard.

Besides the general problem of computing the factor rank, the problem that deserves attention is that of detecting matrices with fixed
factor rank. Certain hard problems of combinatorial optimization admit fast solutions if the input
matrices are required to have factor rank bounded by a fixed number~\cite{Barv, Barv2}.
The Traveling Salesman problem (TSP) also admits a fast solution if we require the distance
matrix to have a fixed factor rank~\cite{Barv3, Barv4}; this special case of TSP is also
known as TSP with warehouses~\cite{Barv3}. Matrices with bounded factor rank arise naturally
in the problems mentioned and in a number of other problems in combinatorial optimization~\cite{Barv} and tropical geometry~\cite{DSS}.
These considerations led to the following interesting question on tropical matrix factorizations.

\begin{quest}\cite{Barv, Barv2}\label{que1}
Does there exist an algorithm that solves TMF for every fixed $k$ in polynomial time?
\end{quest}

In~\cite{Barv}, Barvinok expected that Question~\ref{que1} can be answered in the positive,
this question has also been mentioned in~\cite{Barv2}. Further investigations on the problem
of determining matrices with fixed factor rank have been carried out in~\cite{CRW, Dev1, Dev2, DSS, HJ}.
However, the problem remained open and has been formulated again in~\cite{DSS}, where its connections
with tropical geometry were pointed out. To formulate another problem on the complexity of matrix factorizations
posed in~\cite{DSS}, we define the \textit{tropical rank} of a matrix as the topological dimension of the tropical
linear span of its columns.

\begin{quest}\label{que3}\cite[Section 8, Question 3b]{DSS}
Is there a polynomial-time algorithm for the factor
rank of matrices with bounded tropical rank?
\end{quest}

The progress in solving the problems we have mentioned has mostly been based on studying matrices with
factor rank at most $2$, and the set of such matrices is now indeed well studied.
First, the TMF problem with $k\leq2$ can be solved by a linear-time algorithm, see~\cite{CRW, DSS}.
Further, it has been proven in~\cite{DSS} that the factor rank
of a matrix is at most $2$ if and only if all its $3$-by-$3$ minors have factor ranks
at most $2$. The set of $d$-by-$n$ matrices with factor rank $2$ has been studied
as a simplicial complex in~\cite{Dev1}. This set has also been studied from the topological
point of view in~\cite{HJ}, and the space of $d$-by-$n$ matrices of factor rank two modulo
translation and rescaling has been shown to form a manifold.
For general $k$, the question of fast algorithm for recognizing tropical matrices with factor rank $k$ remained open.


The following notation is used throughout our paper. By $A_{ij}$ or $[A]_{ij}$ we will denote the $(i,j)$th entry of a matrix $A$,
by $A[r_1,\ldots,r_p|c_1,\ldots,c_q]$ the submatrix of $A$ formed by the rows with labels $r_1,\ldots,r_p$ and columns with $c_1,\ldots,c_q$.
We write $B\geq C$ or $C\leq B$ for matrices $B,C\in\R^{m\times n}$ if $B_{ij}\geq C_{ij}$, for any indexes $i$ and $j$.
We will say that matrices $B$ and $C$ coincide modulo \textit{scaling} if there exist real numbers
$\alpha_1,\ldots,\alpha_m$, $\beta_1,\ldots,\beta_n$ such that $B_{ij}=C_{ij}+\alpha_i+\beta_j$ for all $i$ and $j$.
Straightforwardly, $B$ and $C$ have the same factor ranks if they differ only by scaling.

\section{Preliminaries}\label{prelim}

In our paper, we answer Question~\ref{que1}, showing that no algorithm is likely to solve TMF in polynomial time for every fixed $k$.
More precisely, we will show that it is NP-hard to decide whether a given tropical matrix has factor rank at most $7$.
So we are interested in the TMF problem restricted to matrices of fixed rank in our paper, and let us give a precise formulation of
that problem. Our NP-hardness result on the tropical factorization will remain true for matrices consisting of integers, so it will
be convenient for us to assume that the input matrices have integer entries. For a positive integer $k$, the problem we deal with is as follows.

\begin{prob}\label{facrankprob}\label{probmain} TROPICAL MATRIX $k$-FACTORIZATION ($k$-TMF).

Given a tropical matrix $A\in\Z^{m\times n}$.

Question: Do there exist tropical matrices $B\in\R^{m\times k}$
and $C\in\R^{k\times n}$ satisfying $B\ot C=A$?
\end{prob}

We may assume that the input integers in Problem~\ref{probmain} are written in the decimal system.
However, the results we prove and all our considerations remain true even if the integers are written in unary.
Now we will prove that the tropical factorization problem belongs to the class NP
if the input matrices are required to consist of integers, and we need the following technical lemma.

\begin{lem}\label{NP-tmf-lem}
Let a matrix $A\in\Z^{m\times n}$ have factor rank at most $k$. Let $g$ and $l$ denote, respectively,
the greatest and the least elements of $A$, set also $h=|g|+|l|$. Then there exist matrices $B\in\Z^{m\times k}$
and $C\in\Z^{k\times n}$ such that $B\ot C=A$ and $|B_{i\tau}|\leq h$, $|C_{\tau j}|\leq h$ for any triple of indexes $(i, j, \tau)$.
\end{lem}

\begin{proof}
By Definition~\ref{factrank}, there are matrices $B'\in\R^{m\times k}$
and $C'\in\R^{k\times n}$ satisfying $B'\ot C'=A$. Consider matrices $B''$ and $C''$
defined by $B''_{i\tau}=B'_{i\tau}-\{B'_{i\tau}\}$, $C''_{\tau j}=C'_{\tau j}+\{-C'_{\tau j}\}$,
where $\{x\}=x-[x]$ is the fractional part of $x$.
The matrices $B''$ and $C''$ are then integer, and we have $|(B''_{i\tau}+C''_{\tau j})-(B'_{i\tau}+C'_{\tau j})|<1$.
Therefore, $B'_{i\tau}+C'_{\tau j}=A_{ij}$ implies $B''_{i\tau}+C''_{\tau j}=A_{ij}$, and from
$B'_{i\tau}+C'_{\tau j}\geq A_{ij}$ it follows that $B''_{i\tau}+C''_{\tau j}\geq A_{ij}$.
Thus we have that $B'\ot C'=A$ implies $B''\ot C''=A$.

Further, we subtract $b_\tau$, the least element of the $\tau$th column of $B''$, from every entry of the $\tau$th column of $B''$ and
add $b_\tau$ to every entry of the $\tau$th row of $C''$. The matrices $B$ and $C$ obtained satisfy $B\ot C=A$, and zero appears as the
minimal element of every column of $B$. The definition of matrix multiplication then shows that every entry of $C$ is greater than or equal
to $l$. Finally, those entries of $B$ and $C$ that are greater than $|g|+|l|$ can be then replaced by $|g|+|l|$ without changing the product $B\ot C$.
\end{proof}

Now we can prove our first results concerning the computational complexity of tropical matrix factorization.

\begin{thr}\label{NP-tmf}
Given a tropical matrix $A\in\Z^{m\times n}$ and a positive integer $k$. The problem
of deciding whether $A$ has factor rank at most $k$ belongs to the class NP.
\end{thr}

\begin{proof}
The factor rank of $A$ is at most $\min\{m,n\}$ (see~\cite[Proposition~2.1]{DSS}), so the problem can be
solved immediately if $k\geq\min\{m,n\}$. For $k<\min\{m,n\}$, the result follows from Lemma~\ref{NP-tmf-lem}.
\end{proof}

Theorem~\ref{NP-tmf} shows that TMF, the general problem of tropical matrix factorization, belongs to NP if the input
matrices are assumed to consist of integers. The problem $k$-TMF, introduced in this section, is thus in NP as well.

\begin{thr}\label{NP-k-tmf}
The $k$-TMF problem belongs to NP.
\end{thr}

\begin{proof}
Follows from Theorem~\ref{NP-tmf}.
\end{proof}

The goal of our paper is to show that the $k$-TMF problem is NP-complete for any $k\geq7$.
Since $k$-TMF is proven to be in NP in the present section, we now need to construct a polynomial
reduction from some known NP-complete problem to $k$-TMF. The method used for that construction
can be briefly outlined as follows.

Our reduction will use matrices over the \textit{extended tropical semiring} $\Ro=\R\cup\{\i\}$ rather than usual
tropical matrices, and the concept of the extended tropical semiring is introduced in Section~\ref{secext}. It is
also explained in that section why we are allowed to use matrices over $\Ro$ for proving NP-completeness of
factorizations over $\R$. In Section~\ref{secfactint}, we study a relaxed version of the factorization problem that we call
\textit{intermediate factorization}. This relaxed version consists in finding a matrix $M$ which has factor rank not
exceeding $k$ and satisfies $A\leq M\leq B$, for given matrices $A$ and $B$ and an integer $k$. Note that when $A=B$, this problem
is the usual TMF, and we show in Section~\ref{secfactint} that intermediate factorization can in turn be reduced to TMF. The key of our
paper is Section~\ref{secgraphcol}, where we present a reduction to indermediate factorization from a classical NP-complete
problem of the graph colorability. Then, after proving some technical propositions in Section~\ref{secfactB}, we put the
results of previous sections together and prove our main results in Section~\ref{secmain}.

%

\section{Extended tropical semiring}\label{secext}

In our NP-completeness proof for the $k$-TMF problem, we will use the matrices over a certain
extension of the tropical semiring which we call here the \textit{extended tropical semiring}.
Namely, we extend the tropical semiring by an infinite positive element, which we denote by $\i$.
We also write $\Ro$ for $\R\cup\{\i\}$ and assume $a\op\i=a$, $a\ot\i=\i$ for $a\in\Ro$.

Defining the factor rank, one can now think of a tropical matrix $A\in\R^{m\times n}$ as a matrix over $\Ro$, and
allow matrices $B$ and $C$ from Definition~\ref{factrank} to contain infinite elements.
However, it turns out that the rank function defined in this way is
the same as that defined with respect to Definition~\ref{factrank}.
Indeed, if matrices $A\in\R^{m\times n}$, $B\in\Ro^{m\times k}$, $C\in\Ro^{k\times n}$ satisfy $A=B\ot C$,
then we can replace infinite entries of $B$ and $C$ with a sufficiently large real without changing the product of the matrices.

We can therefore extend Definition~\ref{factrank} to the case of matrices over $\Ro$ and define
the \textit{factor rank} of $A\in\Ro^{m\times n}$ to be the smallest integer $k$ for which there exist
$B\in\Ro^{m\times k}$ and $C\in\Ro^{k\times n}$ satisfying $A=B\ot C$. The following technical
lemma will play an important role in the considerations of our paper.

\begin{lem}\label{Ro->R}
Let zero appear as a minimal element of every row and every column of a matrix $A\in\Ro^{m\times n}$.
Consider the matrix $A'$ obtained from $A$ by replacing every infinite entry with the number $2g+1$,
where $g$ stands for the maximal finite entry of $A$. Then the factor ranks of $A$ and $A'$ are the same.
\end{lem}

\begin{proof}
Consider matrices $B'\in\R^{m\times k}$ and $C'\in\R^{k\times n}$ satisfying $B'\ot C'=A'$. For $t\in\{1,\ldots,k\}$,
we subtract $b_t$, the least element of the $t$th column of $B$, from every entry of that column and add $b_t$ to every entry of
the $t$th row of $C$. The matrices $B''$ and $C''$ obtained satisfy $B''\ot C''=A'$, and zero appears as a minimal element
of every row of $B''$ and of every column of $C''$. Further, we replace by $\i$ every entry of $B''$ and $C''$ that is greater
than $g$, and we denote the matrices obtained by $B$ and $C$. It is then easy to check that $B\ot C=A$.

So we have proven that the factor rank of $A$ is at most that of $A'$. It is thus sufficient to note that $A'=E\ot A$, where
$E$ stands for an $m$-by-$m$ tropical matrix satisfying $E_{ij}=0$ if $i=j$ and $E_{ij}=2g+1$ otherwise.
\end{proof}

The following easy observation will also be useful.

\begin{stat}\label{block}
The factor rank of a tropical matrix $A\in\Ro^{m\times n}$ is one less than that of the matrix $A'\in\Ro^{(m+1)\times (n+1)}$ defined by
$A'_{ij}=A_{ij}$, $A'_{i,n+1}=A'_{m+1,j}=\i$ for $i\leq m$, $j\leq n$, and $A'_{m+1,n+1}=0$.
\end{stat}

While the definitions do not provide a direct connection between the factorizations over $\Ro$ and the usual
tropical factorizations, the use of the extended semiring will be helpful for proving the NP-completeness of $k$-TMF.
Note that, given a matrix $A$ over $\Ro$, one can perform an appropriate scaling on $A$ to obtain a matrix $A'$ with
the same factorization rank as $A$ but having zero as the minimal entry of every row and every column. Then one can apply
Lemma~\ref{Ro->R}, thus reducing the problem of matrix factoring over $\Ro$ to that over the usual tropical semiring.

\section{Factoring intermediate matrices}\label{secfactint}

In this section, we study the problem which we call the factorization of an \textit{intermediate matrix}:
Given tropical matrices $A,B\in\Ro^{m\times n}$ satisfying $A\geq B$ and an integer $r$, does there exist a matrix $M$
with factor rank not exceeding $r$ such that $A\geq M\geq B$. Note that this problem turns to the usual problem of factoring
the matrix $A$ (we further call this problem the \textit{exact factorization}) if $B$ is set to equal $A$. This reduces the
exact factorization problem to the intermediate version; the goal of the present section is to prove the opposite result,
reducing the intermediate to exact factorizations.

We will also be interested in the structure of the matrix $B$ mentioned above. Namely, we will assume that the input
of the intermediate factorization problem is an integer $r$ and a triple of matrices $A\in\Ro^{m\times n}$, $C\in\Ro^{m\times k}$,
and $D\in\Ro^{k\times n}$; the goal will be to find a matrix $M$ with factor rank not exceeding $r$ such that $A\geq M\geq A\op(C\ot D)$.
The following definition will be further shown to provide a reduction to the TMF problem.

\begin{defn}\label{definttropfact}
For matrices $A\in\Ro^{m\times n}$, $C\in\Ro^{m\times k}$, and $D\in\Ro^{k\times n}$,
define the matrix $Q=\Q(A,C,D)$ with rows indexed $1',\ldots,k',1,\ldots,m$ and columns
indexed $1',\ldots,k',1,\ldots,n$ as follows. For $\alpha,\beta\in\{1,\ldots,k\}$, $i\in\{1,\ldots,m\}$,
and $j\in\{1,\ldots,n\}$, we set

(q1) $Q_{\alpha'\beta'}=\i$ if $\alpha\neq\beta$ and $Q_{\alpha'\beta'}=0$ otherwise;

(q2) $Q_{\alpha'j}=D_{\alpha j}$;

(q3) $Q_{i\beta'}=C_{i\beta}$;

(q4) $Q[1,\ldots,m|1,\ldots,n]=A\op(C\ot D)$.

In other words, we define $Q$ as the block matrix $\left(
\begin{array}{c|c}
I_k & D \\\hline
C & A\op(C\ot D) \\
\end{array}
\right)$ with $I_k$ standing for the tropical unity $k$-by-$k$ matrix.
\end{defn}

To prove that Definition~\ref{definttropfact} is a reduction, the following auxiliary lemma is needed.

\begin{lem}\label{leminttropfact1}
Consider matrices $A\in\Ro^{m\times n}$, $C\in\Ro^{m\times k}$, $D\in\Ro^{k\times n}$.
If there are matrices $U\in\Ro^{m\times r}$ and $V\in\Ro^{r\times n}$ satisfying
$A\geq U\ot V\geq A\op(C\ot D)$, then the matrix $Q=\Q(A,C,D)$ has factor rank not
exceeding $k+r$.
\end{lem}

\begin{proof}
Define $C'=\left(\begin{array}{c} I_k\\\hline C \end{array}\right)$ and $D'=(I_k|D)$, where
$I_k$ denotes the $k$-by-$k$ matrix with zeros on the diagonal and $\i$'s everywhere else.
Also define $U'=\left(\begin{array}{c} O_{k\times r}\\\hline U \end{array}\right)$ and $V'=(O_{r\times k}|V)$
with $O_{k\times r}$ standing for the $k$-by-$r$ matrix consisting of $\i$'s. To prove the lemma,
it is sufficient to note that $Q=(U'\ot V')\op(C'\ot D')$.
\end{proof}

Let us prove Lemma~\ref{leminttropfact1} in the opposite direction under a certain additional assumption.

\begin{lem}\label{leminttropfact2}
Consider matrices $A\in\Ro^{m\times n}$, $C\in\Ro^{m\times k}$, $D\in\Ro^{k\times n}$
and assume that a submatrix $A[1,\ldots,r|1,\ldots,r]$ has factor rank equal to $r$.
Assume also for every $\kappa\in\{1,\ldots,k\}$, it holds that either $C_{1\kappa}=\ldots=C_{r\kappa}=\i$
or $D_{\kappa1}=\ldots=D_{\kappa r}=\i$. Assume that the matrix $Q=\Q(A,C,D)$ has factor rank not exceeding $k+r$,
then there is a matrix $M$ of factor rank not exceeding $r$ satisfying $M\op(C\ot D)=A\op(C\ot D)$.
\end{lem}

\begin{proof}
\textit{Step~1.} By the assumption of the lemma, there are matrices $U'\in\Ro^{m\times(k+r)}$ and $V'\in\Ro^{(k+r)\times n}$
satisfying $Q=U'\ot V'$. By $Q^{(\tau)}$ denote the product of the $\tau$th column of $U'$ and the $\tau$th
row of $V'$, then $Q=Q^{(1)}\op\ldots\op Q^{(k+r)}$.

\textit{Step~2.} Up to permutations on $(Q^{(1)},\ldots,Q^{(k+r)})$ we can assume that the submatrix $Q^{(\tau)}[1',\ldots,k'|1',\ldots,k']$
has at least one finite entry if $\tau\leq t_0$ and $Q^{(\tau)}[1',\ldots,k'|1',\ldots,k']$ consists of $\i$'s if $\tau>t_0$.

\textit{Step~3.} The assumption of the lemma and Definition~\ref{definttropfact} imply that for any $\kappa\in\{1,\ldots,k\}$, we
have either $Q_{\kappa'1}=\ldots=Q_{\kappa'r}=\i$ or $Q_{1\kappa'}=\ldots=Q_{r\kappa'}=\i$,
so that for any $\tau$ either $Q^{(\tau)}_{\kappa'1}=\ldots=Q^{(\tau)}_{\kappa'r}=\i$ or
$Q^{(\tau)}_{1\kappa'}=\ldots=Q^{(\tau)}_{r\kappa'}=\i$.
Since $Q^{(\tau)}$ has rank one, we see that if for some $\kappa_1,\kappa_2\in\{1,\ldots,k\}$ the condition
$Q^{(\tau)}_{\kappa_1'\kappa_2'}\neq\i$ holds, then for all $\rho_1,\rho_2\in\{1,\ldots,r\}$ it holds that $Q^{(\tau)}_{\rho_1\rho_2}=\i$.
Steps~2 now implies that $Q[1,\ldots,r|1,\ldots,r]=Q^{(t_0+1)}[1,\ldots,r|1,\ldots,r]\op\ldots\op Q^{(k+r)}[1,\ldots,r|1,\ldots,r]$.

\textit{Step~4.} By the assumption of the lemma $[C\otimes D]_{\rho_1\rho_2}=\i$, for all $\rho_1,\rho_2\in\{1,\ldots,r\}$,
so by Definition~\ref{definttropfact} we have $Q[1,\ldots,r|1,\ldots,r]=A[1,\ldots,r|1,\ldots,r]$. Since the latter matrix
has factor rank $r$ by the assumption of the lemma, we apply Step~3 and obtain $t_0\leq k$.

\textit{Step~5.} By Definition~\ref{definttropfact}, the matrix $Q[1',\ldots,k'|1',\ldots,k']$ has zeros on
the diagonal and $\i$'s everywhere else. Steps~1 and~2 imply that $t_0\geq k$, and the result of Step~4 then implies $t_0=k$.
So up to permutations on $(Q^{(1)},\ldots,Q^{(k)})$, we have that $Q^{(\tau)}_{\tau'\tau'}=0$ is a unique finite entry of the matrix
$Q^{(\tau)}[1',\ldots,k'|1',\ldots,k']$.

\textit{Step~6.} By Step~1 we have $Q^{(\tau)}_{i\tau'}\geq Q_{i\tau'}=C_{i\tau}$ and $Q^{(\tau)}_{\tau'j}\geq Q_{\tau'j}=D_{\tau j}$,
for any $i\in\{1,\ldots,m\}$ and $j\in\{1,\ldots,n\}$. Since $Q^{(\tau)}$ is rank-one, we have
$Q^{(\tau)}_{ij}+Q^{(\tau)}_{\tau'\tau'}=Q^{(\tau)}_{i\tau'}+Q^{(\tau)}_{\tau'j}$, so by the result of Step~5
$Q^{(\tau)}_{ij}\geq C_{i\tau}+D_{\tau j}$. Therefore one obtains
$Q^{(1)}[1,\ldots,m|1,\ldots,n]\op\ldots\op Q^{(k)}[1,\ldots,m|1,\ldots,n]\geq C\ot D$, so that
$Q[1,\ldots,m|1,\ldots,n]\op(C\ot D)=(C\ot D)\op M$ with $$M=Q^{(k+1)}[1,\ldots,m|1,\ldots,n]\op\ldots\op Q^{(k+r)}[1,\ldots,m|1,\ldots,n],$$
or by item~(q4) of Definition~\ref{definttropfact}, $A\op(C\ot D)=M\op(C\ot D)$, from which the lemma follows.
\end{proof}

Now we can show that Definition~\ref{definttropfact} reduces intermediate to exact factorizations,
under certain additional assumptions.

\begin{thr}\label{thrinttropfact2}
Consider matrices $A\in\Ro^{m\times n}$, $C\in\Ro^{m\times k}$, $D\in\Ro^{k\times n}$
and assume that a submatrix $A[1,\ldots,r|1,\ldots,r]$ has factor rank equal to $r$.
Assume that for every $i$ and $j$, either $A_{ij}<[C\ot D]_{ij}$ or $A_{ij}=\i$ holds.
Assume also for every $\tau\in\{1,\ldots,k\}$, it holds that either $C_{1\tau}=\ldots=C_{r\tau}=\i$
or $D_{\tau1}=\ldots=D_{\tau r}=\i$. Then the matrix $Q=\Q(A,C,D)$ has factor rank not exceeding $k+r$
if and only if 
there is a matrix $M\in\Ro^{m\times n}$ with factor rank not exceeding $r$ satisfying $A\geq M\geq A\op(C\ot D)$.
\end{thr}

\begin{proof}
Note that under the assumption that either $A_{ij}<[C\ot D]_{ij}$ or $A_{ij}=\i$ holds,
the condition $M\op(C\ot D)=A\op(C\ot D)$ is equivalent to that $A\geq M\geq A\op(C\ot D)$.
Then the theorem follows from Lemmas~\ref{leminttropfact1} and~\ref{leminttropfact2}.
\end{proof}

\section{Reducing graph colorability to intermediate factorizations}\label{secgraphcol}

This is a key section in our NP-completeness proof for the $7$-TMF problem. We will define the
classical NP-complete problem of the graph colorability and then reduce it to the problem of intermediate
tropical factorizations. The problem from which our reduction goes has been considered in~\cite{GJ} and is
formulated as follows.

\begin{prob}\label{probgr3col} GRAPH $3$-COLORABILITY.

Given a simple graph $G$ with vertex set $\{1,\ldots,n\}$ and edges $\{h_1,t_1\},\ldots,\{h_m,t_m\}$.

Question: Does there exist a function $\varphi:\{1,\ldots,n\}\rightarrow\{1,2,3\}$ such that $\varphi(h_j)\neq\varphi(t_j)$ for every $j\in\{1,\ldots,m\}$?
\end{prob}

We will say that the graph $G$ is \textit{$3$-colorable} if the answer in Problem~\ref{probgr3col} is 'yes'.
The following is the classical result of complexity theory.

\begin{thr}\cite{GJ}
Problem~\ref{probgr3col} is NP-complete.
\end{thr}

The question in Problem~\ref{probgr3col} is trivially answered with 'yes' if $G$ has no edges,
so it remains NP-complete if $G$ has at least one edge. Therefore, we can assume in what follows
that the input edges in Problem~\ref{probgr3col} are encoded as a tuple $(h_1,t_1,\ldots,h_m,t_m)$
with $m>0$. We will also call $h_i$ the \textit{head} and $t_i$ the \textit{tail} of the $i$th edge of $G$.

We will reduce Problem~\ref{probgr3col} to the problem of finding a matrix $M$
with factor rank at most three satisfying $A\geq M\geq A\op B$, for input matrices $A$ and $B$.
To start describing our reduction, let us define the set indexed with vertices of $G$ as $$\Vv=\{1\vv,\ldots,n\vv\},$$
and the two sets assigned to the edge set as $$\Hh=\{1\hh,\ldots,m\hh\}\,\,and\,\,\Tt=\{1\tt,\ldots,m\tt\}.$$

\begin{defn}\label{defA}
Define the matrix $\A=\A(G)$ as follows. Let $\{1,2,3\}\cup\Vv\cup\Hh\cup\Tt$ be the set of row indexes,
and $\{1,2,3\}\cup\Hh\cup\Tt$ the set of column indexes. For $\chi\in\{1,2,3\}$, $i\in\{1,\ldots,n\}$, and
$j\in\{1,\ldots,m\}$, we define

(a1) $\A_{\chi\psi}=0$ if $\psi$ from $\{1,2,3\}$ is different from $\chi$;

(a2) $\A_{i\vv, \chi}=0$;

(a3) $\A_{j\hh,j\hh}=40j+1$;

(a4) $\A_{j\hh,j\tt}=\A_{j\tt,j\hh}=40j+9$;

(a5) the entries of $\A$ not defined in items (a1)--(a4) are equal to $\infty$.
\end{defn}

We proceed with the definition of the matrix $\B$.

\begin{defn}\label{defB}
Define the matrix $\B=\B(G)$ as follows. Let $\{1,2,3\}\cup\Vv\cup\Hh\cup\Tt$ be the set of row indexes,
and $\{1,2,3\}\cup\Hh\cup\Tt$ the set of column indexes. Set $H=40m+21$ and for $\psi,\chi\in\{1,2,3\}$,
$i\in\{1,\ldots,n\}$, and $g,j\in\{1,\ldots,m\}$, define

(b1) $\B_{\psi\chi}=\B_{i\vv\chi}=\i$;

(b2) $\B_{\psi j\hh}=\B_{\psi j\tt}=\B_{j\hh \psi}=\B_{j\tt\psi}=0$;

(b3) $\B_{i\vv j\hh}=\left(1-2|i-h_j|\right)H$ and $\B_{i\vv j\tt}=\left(1-2|i-t_j|\right)H$;

(b4) $\B_{g\hh j\hh}=\B_{g\hh j\tt}=\B_{g\tt j\hh}=\B_{g\tt j\tt}=40\min\{g,j\}+20$.
\end{defn}

Throughout the rest of our paper, we write simply 'item~(a1)', $\ldots$, 'item~(a5)' when we need to refer to those items from Definition~\ref{defA};
we adopt the similar convention for Definition~\ref{defB}. As in Definition~\ref{defB}, $H$ will always stand for the number $40m+21$.
We will need the following easy fact on the matrices introduced.

\begin{observ}\label{observ1}
For any indices $\alpha$ and $\beta$, we have $\A_{\alpha\beta}<\B_{\alpha\beta}$ if $\A_{\alpha\beta}<\i$.
\end{observ}

We are now going to prove the relations between the matrices introduced and the GRAPH $3$-COLORABILITY problem.
Namely, we want to show that a graph $G$ admits a $3$-coloring if and only if
\begin{equation}\label{eqcoloring}\A(G)\geq M\geq \A(G)\op\B(G)\end{equation}
holds for some matrix $M$ of factor rank not exceeding three. We start with the 'only if' part of this statement.

\begin{lem}\label{lemonlyif}
Let $G$ be a $3$-colorable graph. Then there is a matrix $M$ of factor rank
not exceeding $3$ that satisfies~\eqref{eqcoloring}.
\end{lem}

\begin{proof}
By the assumption, there is a function $\varphi:\{1,\ldots,n\}\rightarrow\{1,2,3\}$ such that $\varphi(h_j)\neq\varphi(t_j)$, for all $j\in\{1,\ldots,m\}$.
Define the matrix $U$ with rows indexed $\{1,2,3\}\cup\Vv\cup\Hh\cup\Tt$ and columns indexed $\{1,2,3\}$ as follows. For $\psi,\chi\in\{1,2,3\}$,
$i\in\{1,\ldots,n\}$, and $j\in\{1,\ldots,m\}$, we set

(u1) $U_{\psi\chi}=\i$ if $\psi\neq\chi$ and $U_{\psi\chi}=0$ if $\psi=\chi$;

(u2) $U_{i\vv \chi}=\i$ if $\varphi(i)=\chi$ and $U_{i\vv \chi}=0$ if $\varphi(i)\neq\chi$;

(u3) $U_{j\hh,\varphi(h_j)}=20j$, $U_{j\hh,\varphi(t_j)}=20j+8$, and $U_{j\hh,\chi}=\i$ if $\chi\notin\{\varphi(h_j),\varphi(t_j)\}$;

(u4) $U_{j\tt,\varphi(h_j)}=20j+8$ and $U_{j\tt,\chi}=\i$ if $\chi\neq\varphi(h_j)$.

To define the matrix $V$ with rows indexed $\{1,2,3\}$ and columns indexed $\{1,2,3\}\cup\Hh\cup\Tt$,
for $\psi,\chi\in\{1,2,3\}$, $i\in\{1,\ldots,n\}$, and $j\in\{1,\ldots,m\}$, we set

(v1) $V_{\psi\chi}=0$ if $\psi\neq\chi$ and $V_{\psi\chi}=\i$ if $\psi=\chi$;

(v2) $V_{\psi j\hh}=20j+1$ if $\psi=\varphi(h_j)$ and $V_{\psi j\hh}=\i$ if $\psi\neq\varphi(h_j)$;

(v3)  $V_{\psi j\tt}=20j+1$ if $\psi=\varphi(t_j)$ and $V_{\psi j\tt}=\i$ if $\psi\neq\varphi(t_j)$.

Let us now check that matrices $U$ and $V$ satisfy~\eqref{eqcoloring}, that is,
\begin{equation}\label{eqcoloring153}\A_{\alpha\beta}\geq \min_{\chi=1}^3 \{U_{\alpha\chi}+V_{\chi\beta}\}
\geq\min\{\A_{\alpha\beta},\B_{\alpha\beta}\}\end{equation}
holds for any $\alpha\in\{1,2,3\}\cup\Vv\cup\Hh\cup\Tt$ and $\beta\in\{1,2,3\}\cup\Hh\cup\Tt$.

\textit{Step~1.} When $(\alpha,\beta)\in\{1,2,3\}\times\left(\Hh\cup\Tt\right)$ or $(\alpha,\beta)\in\left(\Hh\cup\Tt\right)\times\{1,2,3\}$,
then we have $\A_{\alpha\beta}=\i$ by item (a5) and $\B_{\alpha\beta}=0$ by item (b2).
The condition~\eqref{eqcoloring153} is then immediate since the entries of $U$ and $V$ are nonnegative.

\textit{Step~2.} To check~\eqref{eqcoloring153} for $\alpha\in\{1,2,3\}\cup\Vv$ and $\beta\in\{1,2,3\}$,
it is enough to prove that $$\A[1,2,3,i\vv|1,2,3]=U[1,2,3,i\vv|1,2,3]\ot V[1,2,3|1,2,3],$$
that is,
$$\left(\begin{array}{ccc} \i & 0 &0\\ 0 & \i &0\\ 0& 0 & \i\\0&0&0 \end{array}\right)=
\left(\begin{array}{ccc} 0 & \i &\i\\ \i &0 &\i\\ \i& \i &0\\U_{i\vv1}&U_{i\vv2}&U_{i\vv3} \end{array}\right)\ot
\left(\begin{array}{ccc} \i & 0 &0\\ 0 & \i &0\\ 0& 0 & \i \end{array}\right),$$
which is true because by item~(u2) the two entries of $U_{i\vv1},U_{i\vv2},U_{i\vv3}$ equal zero.

\textit{Step~3.} Note that by item (a5) $\A_{i\vv j\hh}=\i$. To check~\eqref{eqcoloring153} for $\alpha=i\vv$ and $\beta=j\hh$,
it is therefore enough to prove for any $\chi$ that $U_{i\vv \chi}+V_{\chi j\hh}\geq \B_{i\vv j\hh}$. This is easy when $h_j\neq i$
because then $\B_{i\vv j\hh}<0$ by item (b3). On the other hand, if $h_j=i$, then by items~(u2) and~(v2) we have $U_{i\vv \chi}+V_{\chi j\hh}=\i$.

Similarly, for $\alpha=i\vv$ and $\beta=j\tt$, it is enough to prove for any $\chi$ that $U_{i\vv \chi}+V_{\chi j\tt}\geq \B_{i\vv j\tt}$.
If $t_j\neq i$ we have again $\B_{i\vv j\tt}<0$, so we are done; on the other hand, if $t_j=i$,
items~(u2) and~(v3) again imply $U_{i\vv \chi}+V_{\chi j\hh}=\i$.

Thus we have checked~\eqref{eqcoloring153} for any $(\alpha,\beta)\in\Vv\times\left(\Hh\cup\Tt\right)$.

\textit{Step~4.} Now assume $\alpha\in\{g\hh,g\tt\}$ and $\beta\in\{j\hh,j\tt\}$.

4.1. Then we have by items~(u3) and~(u4) that $U_{\alpha\chi}\geq 20g$, and by items~(v2)
and~(v3) that $V_{\chi\beta}\geq 20j+1$, for $\chi\in\{1,2,3\}$.

4.2. Assume $g\neq j$. Then we have $\A_{\alpha\beta}=\i$ by item (a5) and by paragraph~4.1
$U_{\alpha\chi}+V_{\chi\beta}\geq40\min\{j,g\}+21$, that is, $U_{\alpha\chi}+V_{\chi\beta}> B_{\alpha\beta}$ by item (b4).
The latter inequality shows that~\eqref{eqcoloring153} holds.

4.3. Assume $\alpha=\beta=g\tt$. Then we have $U_{\alpha\chi}+V_{\chi\beta}=\i$ by items~(u4) and~(v3) since $\varphi(h_j)\neq\varphi(t_j)$.
Also $\A_{\alpha\beta}=\i$ in this case, so~\eqref{eqcoloring153} holds.

4.4. If $g=j$ but the assumption of item~4.3 does not hold, then we check that $\A_{\alpha\beta}=\min_{\chi=1}^3 \{U_{\alpha\chi}+V_{\chi\beta}\}$
and conclude that~\eqref{eqcoloring153} holds.

Items 4.2--4.4 cover all the possibilities within Step~4, so we have checked~\eqref{eqcoloring153}
for any $(\alpha,\beta)\in\left(\Hh\cup\Tt\right)\times\left(\Hh\cup\Tt\right)$. Finally, note that
Steps~1--4 cover all the possibilities for $(\alpha,\beta)$, so our proof is complete.
\end{proof}

To prove the result of Lemma~\ref{lemonlyif} in the opposite direction, we need the following auxiliary statement.

\begin{lem}\label{lem000}
Assume the columns of a matrix $V\in\Ro^{3\times v}$ indexed $1$, $2$, and $3$ form a submatrix
$\left(\begin{smallmatrix}0 & \i &\i\\ \i & 0 &\i\\ \i& \i & 0\end{smallmatrix}\right)$.
Let $U\in\Ro^{u\times3}$ and assume that the matrix $U\ot V$ has nonnegative elements
in columns indexed $1$, $2$, and $3$. If there is a $g$ satisfying $[U\ot V]_{g\chi}=0$
for any $\chi\in\{1,2,3\}$, then $[U\ot V]_{ij}\geq[U\ot V]_{gj}$ holds for any indexes $i$ and $j$.
\end{lem}

\begin{proof}
The assumption of the lemma implies
$$\left(U_{i1}\ot(0\, \i \,\i)\right)\op\left(U_{i2}\ot(\i\,0\,\i)\right)\op\left(U_{i3}\ot(\i\,\i\,0)\right)\geq(0\,0\,0)$$
with equality if $i=g$. So we see that $U_{g\chi}=0$ and $U_{i\chi}\geq0$, for any $i$ and $\chi$. Thus we have
$\min_{\chi=1}^3\{U_{i\chi}+V_{\chi j}\}\geq\min_{\chi=1}^3\{U_{g\chi}+V_{\chi j}\}$, for any $i$ and $j$, which is the result we need.
\end{proof}

We need another technical lemma. A tropical matrix $A\in\Ro^{k\times k}$ is \textit{monomial}
if there is a permutation $\sigma$ on $\{1,\ldots,k\}$ such that $A_{ij}=\i$ if and only if $j\neq\sigma(i)$.

\begin{lem}\label{lemeasy0}
If a tropical $3$-by-$k$ matrix $U'$ and a tropical $k$-by-$3$ matrix $V'$ satisfy
$U'\ot V'=\left(\begin{smallmatrix} \i & 0 &0\\ 0 & \i &0\\ 0& 0 & \i \end{smallmatrix}\right)$,
then $k\geq3$. Moreover, if $k=3$, then then either $U'$ or $V'$ is monomial.
\end{lem}

\begin{proof}
Let $\mathrm{Sup}(U',i)$ denote the set of all $j$ satisfying $U'_{ij}\neq\i$. By definition of tropical multiplication,
the condition $\mathrm{Sup}(U',i)\subset\mathrm{Sup}(U',i')$ implies $\mathrm{Sup}(U'\ot V',i)\subset\mathrm{Sup}(U'\ot V',i')$.
So we see that $\mathrm{Sup}(U',1)$, $\mathrm{Sup}(U',2)$, $\mathrm{Sup}(U',3)$ is an antichain under inclusion. This proves the
first assertion of the lemma, and in the case of $k=3$ one can also note that either $|\mathrm{Sup}(U',i)|=1$ for all $i$ or
$|\mathrm{Sup}(U',i)|=2$ for all $i$. In the former case, $U'$ is monomial, and in the latter case $V'$ can be checked to be monomial.
\end{proof}

Let us find the factor rank of an important submatrix of $\A(G)$.

\begin{lem}\label{lemsubA}
The factor rank of the submatrix $\A[1,2,3|1,2,3]$ equals $3$.
\end{lem}

\begin{proof}
By items~(a1) and~(a5), the submatrix discussed equals $\left(\begin{smallmatrix} \i & 0 &0\\ 0 & \i &0\\ 0& 0 & \i \end{smallmatrix}\right)$,
so the result follows from Lemma~\ref{lemeasy0}.
\end{proof}

Let us use Lemmas~\ref{lem000} and~\ref{lemeasy0} to obtain a deeper characterization.

\begin{lem}\label{lem111}
Assume that there is a matrix $M$ of factor rank three satisfying~\eqref{eqcoloring}.
Then there is such $M$ of the form $U\ot V$ where the submatrix of $V$ formed by the columns indexed $1$, $2$, $3$
equals $A'=\left(\begin{smallmatrix} \i & 0 &0\\ 0 & \i &0\\ 0& 0 & \i \end{smallmatrix}\right)$.
\end{lem}

\begin{proof}
Denote by $V'$ the submatrix of $V$ formed by the columns indexed $1$, $2$, $3$, and by
$U'$ the submatrix of $U$ formed by the rows indexed $1$, $2$, $3$.

\textit{Step~1.} We can add a real $r$ to every element of $i$th column of $U$ and
$-r$ to the $i$th row of $V$ without changing the product $U\ot V$. Similarly, we can act with
a permutation $\sigma$ simultaneously on the column labels of $U$ and on the row labels of $V$
without changing the product $U\ot V$.

\textit{Step~2.} By item (b1) we have $B_{\chi\psi}=\i$ for $\chi,\psi\in\{1,2,3\}$,
and by items (a1) and (a5) the matrix formed by the rows and columns of $A$ indexed $1$, $2$, $3$ equals $A'$.
Therefore, the matrix $U'\ot V'$ equals $A'$ by~\eqref{eqcoloring} as well.

\textit{Step~3.} If the submatrix $U'$ is monomial, then we apply Step~1 and conclude that $U'$
equals $\left(\begin{smallmatrix}0 & \i &\i\\ \i &0 &\i\\ \i& \i &0 \end{smallmatrix}\right)$
the tropical unity matrix. Then $V'=U'\ot V'$, which is $A'$ by Step~3, so we are done.

\textit{Step~4.} Assume the submatrix $V'$ is monomial. Then we apply Step~1 again and conclude that
$V'$ equals $\left(\begin{smallmatrix}0 & \i &\i\\ \i &0 &\i\\ \i& \i &0 \end{smallmatrix}\right)$.
Further, by items (a2) and (b1), we have $A_{i\vv\chi}=0$ and $B_{i\vv\chi}=\i$ for any $\chi\in\{1,2,3\}$; so that
by~\eqref{eqcoloring} we have $[U\ot V]_{i\vv\chi}=0$. Apply Lemma~\ref{lem000} to conclude that
$[U\ot V]_{1\hh 1\hh}\geq[U\ot V]_{h_1\vv 1\hh}$. From~\eqref{eqcoloring} it then follows
$[U\ot V]_{1\hh 1\hh}\geq \B_{h_1\vv 1\hh}$, and then from item (b3) $[U\ot V]_{1\hh 1\hh}\geq H$.
This is a contradiction with~\eqref{eqcoloring}
because by item (a3) $A_{1\hh 1\hh}=41<H$, so the assumption of Step~4 is not realizable.

Since $U'\ot V'=A'$ by Step~2, Lemma~\ref{lemeasy0} shows that Steps~3 and~4 cover all the possibilities and complete the proof.
\end{proof}

We are now ready to prove the result opposite to that of Lemma~\ref{lemonlyif}.

\begin{lem}\label{lemifif}
Let a graph $G$ be not $3$-colorable. Then no matrix $M$ of factor rank at most $3$ satisfies~\eqref{eqcoloring}.
\end{lem}

\begin{proof}
Assume the converse. Then there are matrices $U$ with rows indexed $\{1,2,3\}\cup\Vv\cup\Hh\cup\Tt$ and columns indexed $\{1,2,3\}$
and $V$ with rows indexed $\{1,2,3\}$ and columns indexed $\{1,2,3\}\cup\Hh\cup\Tt$ satisfying
\begin{equation}\label{eqcoloring2}\A(G)\geq U\ot V\geq \A(G)\op\B(G).\end{equation}

\textit{Step~1.} By Lemma~\ref{lem111} we can assume that the submatrix of $V$ formed by the columns indexed $1$, $2$, $3$
equals $\left(\begin{smallmatrix} \i & 0 &0\\ 0 & \i &0\\ 0& 0 & \i \end{smallmatrix}\right)$.

\textit{Step~2.} If the row of $U$ with index $\alpha$ contains a negative entry, then by Step~1, there is a $\psi\in\{1,2,3\}$
such that $[U\ot V]_{\alpha\psi}<0$. This contradicts~\eqref{eqcoloring2} because both $\A$ and $\B$ have nonnegative numbers
in columns indexed $1$, $2$, and $3$. Therefore, the entries of $U$ are nonnegative.

\textit{Step~3.} Fix an $i\in\{1,\ldots,n\}$. For any $\chi\in\{1,2,3\}$, we then have $\A_{i\vv\chi}=0$ by item (a2)
and $\B_{i\vv\chi}=\i$ by item (b1).
Now~\eqref{eqcoloring2} implies that $[U\ot V]_{i\vv\chi}=0$. Taking into an account the result of Step~1, one has
$$\left(U_{i\vv1}\ot(\i\,0 \,0)\right)\op\left(U_{i\vv2}\ot(0\,\i\,0)\right)\op\left(U_{i\vv3}\ot(0\,0\,\i)\right)=(0\,0\,0).$$

\textit{Step~4.} The result of Step~3 implies that $U_{i\vv\psi'}=U_{i\vv\psi''}=0$, for some different $\psi'$ and $\psi''$ from $\{1,2,3\}$.
Denote by $\varphi(i)$ the element of $\{1,2,3\}\setminus\{\psi',\psi''\}$. Note that $\varphi$ is a function from $\{1,\ldots,n\}$ to $\{1,2,3\}$.

\textit{Step~5.} Since the graph $G$ is not $3$-colorable, $G$ has an edge $\{h_j,t_j\}$ satisfying $\varphi(h_j)=\varphi(t_j)$;
we take $\{\chi_1,\chi_2\}=\{1,2,3\}\setminus\{\varphi(h_j)\}$.

\textit{Step~6.} Now let us note that $\A_{h_j\vv j\hh}=\i$ by item (a5) and $\B_{h_j\vv j\hh}=H$ by item (b3). Therefore~\eqref{eqcoloring2} implies
$$U_{h_j\vv \chi_1}+V_{\chi_1 j\hh}\geq H.$$ By the results of Steps~4 and~5, $U_{h_j\vv \chi_1}=0$, so that $V_{\chi_1 j\hh}\geq H.$
Similarly we have $V_{\chi_2 j\hh}\geq H.$

\textit{Step~7.} Similarly to Step~6, one has $\A_{t_j\vv j\tt}=\i$ and $\B_{t_j\vv j\tt}=H$. Therefore~\eqref{eqcoloring2} implies
$$U_{t_j\vv \chi_1}+V_{\chi_1 j\tt}\geq H.$$ By the results of Steps~4 and~5, $U_{t_j\vv \chi_1}=0$, so that $V_{\chi_1 j\tt}\geq H$,
and similarly $V_{\chi_2 j\tt}\geq H.$

\textit{Step~8.} Now let us restrict our attention to the submatrix of $U\ot V$ formed by the rows and columns indexed $j\hh$ and $j\tt$.
By~\eqref{eqcoloring2} we have
$$\left(
\begin{array}{cc}
40j+1 & 40j+9 \\
40j+9 & \i \\
\end{array}
\right)
\geq
\left(
\begin{array}{ccc}
U_{j\hh \chi_1} & U_{j\hh \chi_2} & U_{j\hh \varphi(h_j)} \\
U_{j\tt \chi_1} & U_{j\tt \chi_2} & U_{j\tt \varphi(h_j)} \\
\end{array}
\right)
\ot
\left(
\begin{array}{cc}
V_{\chi_1 j\hh} & V_{\chi_1 j\tt} \\
V_{\chi_2 j\hh} & V_{\chi_2 j\tt} \\
V_{\varphi(h_j) j\hh} & V_{\varphi(h_j) j\tt}\\
\end{array}
\right)\geq$$
$$\geq\left(
\begin{array}{cc}
40j+1 & 40j+9 \\
40j+9 & 40j+20 \\
\end{array}
\right).$$

\textit{Step~9.} By Steps~2 and~6, one has
$$\min\{U_{j\hh \chi_1}+V_{\chi_1 j\hh},U_{j\hh \chi_2}+V_{\chi_2 j\hh},U_{j\tt \chi_1}+V_{\chi_1 j\hh},U_{j\tt \chi_2}+V_{\chi_2 j\hh}\}\geq H.$$
From Steps~2 and~7 it also follows that
$$\min\{U_{j\hh \chi_1}+V_{\chi_1 j\tt},U_{j\hh \chi_2}+V_{\chi_2 j\tt},U_{j\tt \chi_1}+V_{\chi_1 j\tt},U_{j\tt \chi_2}+V_{\chi_2 j\tt}\}\geq H.$$

\textit{Step~10.} Noting that $H>40j+20$ by definition, we combine the results of Steps~8 and~9. Then we see
that the tropical rank-one matrix $$M'=\left(
\begin{array}{ccc}
U_{j\hh \varphi(h_j)} \\
U_{j\tt \varphi(h_j)} \\
\end{array}
\right)
\ot
\left(
\begin{array}{cc}
V_{\varphi(h_j) j\hh} & V_{\varphi(h_j) j\tt}\\
\end{array}
\right)$$
satisfies
$$\left(
\begin{array}{cc}
40j+1 & 40j+9 \\
40j+9 & \i \\
\end{array}
\right)
\geq
M'
\geq\left(
\begin{array}{cc}
40j+1 & 40j+9 \\
40j+9 & 40j+20 \\
\end{array}
\right),$$
which is a contradiction.
\end{proof}

The following theorem outlines the results of this section.

\begin{thr}\label{thrifonlyif}
A graph $G$ is $3$-colorable if and only if there is a matrix $M$ of factor rank at most $3$ satisfying $\A(G)\geq M\geq \A(G)\op\B(G)$.
\end{thr}

\begin{proof}
Follows immediately from Lemmas~\ref{lemonlyif} and~\ref{lemifif}.
\end{proof}

\section{Factoring the matrix $\B(G)$}\label{secfactB}

To finalize the NP-completeness proof, we want to combine the results of Theorems~\ref{thrinttropfact2}
and~\ref{thrifonlyif}. However, Theorem~\ref{thrinttropfact2} requires the matrix $\B$ to be presented
as a product of two tropical matrices, so we can not apply this theorem directly. The goal of this section
is to provide such a factorization for the matrix $\B(G)$. We start with the definition of the one of the factors.

\begin{defn}\label{defC}
Define the matrix $\C=\C(G)$ as follows. Let $\{1,2,3\}\cup\Vv\cup\Hh\cup\Tt$ be the set of row indexes,
and $\{1,2,3,4\}$ the set of column indexes. Take

(c1) $\C_{\alpha1}=0$ for $\alpha\in\Hh\cup\Tt$, and $\C_{\alpha1}=\i$ otherwise;

(c2) $\C_{\chi2}=0$ for $\chi\in\{1,2,3\}$ and $\C_{i\vv2}=\i$ for $i\in\{1,\ldots,n\}$;

(c3) $\C_{j\hh2}=\C_{j\tt2}=40j+20$ for $j\in\{1,\ldots,m\}$;

(c4) $\C_{\alpha3}=(2i+1)H$ if $\alpha=i\vv$, and $\C_{\alpha3}=\i$ otherwise;

(c5) $\C_{\alpha4}=(2n+1-2i)H$ if $\alpha=i\vv$, and $\C_{\alpha4}=\i$ otherwise.
\end{defn}

We proceed with the definition of the other factor.

\begin{defn}\label{defD}
Define the matrix $\D=\D(G)$ as follows. Let $\{1,2,3,4\}$ be the set of row indexes,
and $\{1,2,3\}\cup\Hh\cup\Tt$ the set of column indexes. Take

(d1) $\D_{1\chi}=0$ for $\chi\in\{1,2,3\}$;

(d2) $\D_{1j\hh}=\D_{1j\tt}=40j+20$ for $j\in\{1,\ldots,m\}$;

(d3) $\D_{2\chi}=\D_{3\chi}=\D_{4\chi}=\i$ for $\chi\in\{1,2,3\}$;

(d4) $\D_{2j\hh}=\D_{2j\tt}=0$ for $j\in\{1,\ldots,m\}$;

(d5) $\D_{3j\hh}=-2h_jH$ and $\D_{3j\tt}=-2t_jH$ for $j\in\{1,\ldots,m\}$;

(d6) $\D_{4j\hh}=2(h_j-n)H$ and $\D_{4j\tt}=2(t_j-n)H$ for $j\in\{1,\ldots,m\}$.
\end{defn}

Let us point out an easy fact on the matrices introduced.

\begin{observ}\label{observ2}
For every $\tau\in\{1,\ldots,4\}$, it holds that either $\C_{1\tau}=\C_{2\tau}=\C_{3\tau}=\i$
or $\D_{\tau 1}=\D_{\tau 2}=\D_{\tau 3}=\i$.
\end{observ}

Let us now show that $\C(G)$ and $\D(G)$ provide the factorization of the matrix $\B(G)$ from Definition~\ref{defB}.

\begin{lem}\label{lemBfact}
Given a graph $G$, it holds that $\B(G)=\C(G)\otimes \D(G)$.
\end{lem}

\begin{proof}
Consider the five possibilities for $\alpha\in\{1,2,3\}\cup\Vv\cup\Hh\cup\Tt$ and $\beta\in\{1,2,3\}\cup\Hh\cup\Tt$.

\textit{Step~1.} For $\beta\in\{1,2,3\}$, we have $\D_{\tau\beta}=0$ if $\tau=1$ and $\D_{\tau\beta}=\i$ otherwise.
Therefore $[\C\ot \D]_{\alpha\beta}=\C_{\alpha1}$, that is, $[\C\ot \D]_{\alpha\beta}=0$ if $\alpha\in\Hh\cup\Tt$ and $[\C\ot \D]_{\alpha\beta}=\i$ otherwise.
Checking items (b1) and (b2), we obtain $[\C\ot \D]_{\alpha\beta}=\B_{\alpha\beta}$.

\textit{Step~2.} For $\alpha\in\{1,2,3\}$ and $\beta\in\Hh\cup\Tt$, we have $\C_{\alpha\tau}=0$ if $\tau=2$ and $\C_{\alpha\tau}=\i$ otherwise.
Therefore $[\C\ot \D]_{\alpha\beta}=\D_{2\beta}=0$. On the other hand, $\B_{\alpha\beta}$ in this case equals $0$ as well by item (b2).

\textit{Step~3.} For $\alpha,\beta\in\Hh\cup\Tt$, we have $\alpha\in\{g\hh,g\tt\}$ and $\beta\in\{j\hh,j\tt\}$, for some $g,j\in\{1,\ldots,m\}$.
In this case we have $\C_{\alpha3}=\C_{\alpha4}=\i$, so that
$[\C\ot \D]_{\alpha\beta}=\min\{\C_{\alpha1}+\D_{1\beta},\C_{\alpha2}+\D_{2\beta}\}=\min\{40g+20,40j+20\}$.
Comparing the latter result with item (b4), we obtain $[\C\ot \D]_{\alpha\beta}=\B_{\alpha\beta}$.

\textit{Step~4.} Assume $\alpha=i\vv$ and $\beta=j\hh$, for some $i\in\{1,\ldots,n\}$ and $j\in\{1,\ldots,m\}$.
In this case one has $\C_{\alpha1}=\C_{\alpha2}=\i$, so that
$$[\C\ot \D]_{\alpha\beta}=\min\{\C_{\alpha3}+\D_{3\beta},\C_{\alpha4}+\D_{4\beta}\}=\min\{(2i-2h_j+1)H,(2h_j-2i+1)H\}=H-2|i-h_j|H.$$
Comparing the latter result with item (b3), we obtain $[\C\ot \D]_{\alpha\beta}=\B_{\alpha\beta}$.

\textit{Step~5.} Assume $\alpha=i\vv$ and $\beta=j\tt$, for some $i\in\{1,\ldots,n\}$ and $j\in\{1,\ldots,m\}$.
In this case one has $\C_{\alpha1}=\C_{\alpha2}=\i$, so that
$$[\C\ot \D]_{\alpha\beta}=\min\{\C_{\alpha3}+\D_{3\beta},\C_{\alpha4}+\D_{4\beta}\}=\min\{(2i-2t_j+1)H,(2t_j-2i+1)H\}=H-2|i-t_j|H.$$
Comparing the latter result with item (b3), we obtain $[\C\ot \D]_{\alpha\beta}=\B_{\alpha\beta}$.

Steps~1--5 cover all the possibilities for indexes $\alpha$ and $\beta$, so we have $\C\ot \D=\B$.
\end{proof}

We finalize the section with a corollary of lemma proven.

\begin{cor}\label{corobserv}
Given a graph $G$. For any indices $\alpha$ and $\beta$, we have $\A_{\alpha\beta}<[\C(G)\otimes \D(G)]_{\alpha\beta}$ if $\A_{\alpha\beta}<\i$.
\end{cor}

\begin{proof}
Follows immediately from Observation~\ref{observ1} and Lemma~\ref{lemBfact}.
\end{proof}

\section{Main results}\label{secmain}

Now we are ready to put our results together and prove that the $7$-TMF problem is NP-complete.

\begin{thr}\label{7-tmf}
The $7$-TMF problem is NP-complete.
\end{thr}

\begin{proof}
Assume $G$ is a simple graph and consider matrices $\A(G)$ from Definition~\ref{defA}, $\C(G)$ from Definition~\ref{defC},
and $\D(G)$ from Definition~\ref{defD}. Then Lemma~\ref{lemsubA}, Observation~\ref{observ2}, and Corollary~\ref{corobserv} show that these matrices satisfy the
assumptions of Theorem~\ref{thrinttropfact2}. So we see that the matrix $Q(G)=\Q(\A(G),\C(G),\D(G))$ has factor
rank not exceeding $7$ if and only if there are matrices $U\in\Ro^{m\times 3}$ and $V\in\Ro^{3\times n}$ satisfying
$\A(G)\geq U\ot V\geq \A(G)\op(\C(G)\ot \D(G))$. By Lemma~\ref{lemBfact}, $\B(G)=\C(G)\ot \D(G)$, so by Theorem~\ref{thrifonlyif}
$Q(G)$ has factor rank not exceeding $7$ if and only if $G$ is $3$-colorable. Clearly, the matrix $Q(G)$ can be constructed
in polynomial time, given a graph $G$. We finally perform an appropriate scaling and apply Lemma~\ref{Ro->R}, thus giving a
polynomial transformation from Problem~\ref{probgr3col} to $7$-TMF. That $7$-TMF is in NP follows from Theorem~\ref{NP-k-tmf}.
\end{proof}

As a corollary, we can show that the $k$-TMF problem is NP-complete for any integer $k$ greater than or equal to $7$.

\begin{thr}\label{88-tmf}
The $k$-TMF problem is NP-complete for any integer $k\geq7$.
\end{thr}

\begin{proof}
Given a matrix $M$, we use Proposition~\ref{block} to construct a matrix $M'$ which has factor rank not exceeding $k$
if and only if the factor rank of $M$ does not exceed $7$. To finalize a reduction from $7$-TMF to $k$-TMF, it remains
to perform an appropriate scaling on $M'$ and then use Lemma~\ref{Ro->R}.
\end{proof}

%
%

Another result that can be obtained as a corollary of Theorem~\ref{7-tmf} gives an answer for Question~\ref{que3}.
In contrast with the situation of the Gondran--Minoux rank, which can be computed in polynomial time if a given
matrix has bounded tropical rank~\cite{our}, Question~\ref{que3} is answered in the negative.

\begin{thr}\label{troprankhard}
Given a matrix $A$ of tropical rank at most $7$, it is NP-hard to decide whether the factor rank of $A$ is at most $7$.
\end{thr}

\begin{proof}
Using the combinatorial definition of tropical rank provided in~\cite{DSS}, one can check in time $O(m^8n^8)$
whether the tropical rank of an $m$-by-$n$ matrix is greater than $7$ or not. Further, Theorem 1.4 of~\cite{DSS} shows that
the factor rank of a matrix is greater than $7$ if the tropical rank is. Thus the present theorem follows from Theorem~\ref{7-tmf}.
\end{proof}

\end{document}